\documentclass[a4paper,12pt]{article}
\usepackage{t1enc}
\usepackage[latin2]{inputenc}
\usepackage{amsfonts}
\usepackage{amsmath}
\usepackage{mathrsfs}
\usepackage{amssymb}
\usepackage{amsthm}
\usepackage{indentfirst}
\usepackage{enumitem}
\usepackage[a4paper,left=3cm,right=3cm,top=2.5cm,bottom=2.5cm]{geometry}

\title{A Tits alternative for topological full groups}
\author{N\'ora Gabriella Sz\H{o}ke\thanks{Institute of Mathematics, EPFL, Lausanne, Switzerland. Email: nora.szoke@epfl.ch.}}
\date{}

\theoremstyle{plain}
\newtheorem{thm}{Theorem}
\newtheorem{theorem}{Theorem}[section]
\newtheorem{lemma}[theorem]{Lemma}
\newtheorem{claim}[theorem]{Claim}
\newtheorem{corollary}[theorem]{Corollary}
\newtheorem{proposition}[theorem]{Proposition}

\theoremstyle{definition}
\newtheorem{definition}[theorem]{Definition}
\newtheorem{notation}[theorem]{Notation}

\theoremstyle{remark}

\newtheorem{remark}[theorem]{Remark}

\newcommand{\acts}{\curvearrowright}
\newcommand{\hooklongrightarrow}{\mathrel{\lhook\mkern -3.5mu\relbar\mkern -4.5mu \rightarrow }}

\begin{document}

\maketitle

\begin{abstract}
We prove a Tits alternative for topological full groups of minimal actions of finitely generated groups. On the one hand, we show that topological full groups of minimal actions of virtually cyclic groups are amenable. By doing so, we generalize the result of Juschenko and Monod for $\mathbf{Z}$-actions. On the other hand, when a finitely generated group $G$ is not virtually cyclic, then we construct a minimal free action of $G$ on a Cantor space such that the topological full group contains a non-abelian free group.
\end{abstract}

\section{Introduction}

\noindent Let $G$ be a group and consider an action of $G$ on a compact topological space $C$ by homeomorphisms. This action is called \emph{minimal} if $C$ has no proper $G$-invariant closed subset. The \emph{topological full group} $[[G \acts C]]$ is defined to be the group of all homeomorphisms of $C$ that are piecewise given by elements of $G$, where each piece is open in $C$.

Topological full groups were introduced by Giordano, Putnam and Skau \cite{GPS99} for $\mathbf{Z}$-actions. Matui investigated these groups in a series of papers (\cite{Mat06}, \cite{Mat11}, \cite{Mat15}). He showed that the derived subgroup of the topological full group is often simple, and in some cases, e.g. for minimal subshifts, it is also finitely generated. Nekrashevych further generalized these results in \cite{Nek17}.

Juschenko and Monod proved that for any minimal action of $\mathbf{Z}$ on a Cantor space $\Sigma$, the topological full group $[[ \mathbf{Z} \acts \Sigma ]]$ is amenable \cite{JM13}. Relying on the results of Matui, they provided the first examples of finitely generated infinite simple amenable groups.

It is natural to ask whether the Juschenko-Monod theorem holds for minimal actions of other (necessarily amenable) groups as well. The result of Elek and Monod \cite{EM13} answers this question, showing that already in the case of $\mathbf{Z}^2$ there exists a counterexample. Their result is even stronger, they construct a minimal action of $\mathbf{Z}^2$ on a Cantor space such that the topological full group contains a non-abelian free group.

The goal of this paper is to generalize both results (\cite{JM13} and \cite{EM13}). We show that the Juschenko-Monod result holds for virtually cyclic groups and any compact space. On the other hand when the group $G$ is infinite but not virtually $\mathbf{Z}$, then we construct a minimal action of $G$ on a Cantor space such that the topological full group contains a non-abelian free group. These two statements form a Tits alternative for all finitely generated groups.

\begin{thm}\label{thm:A}
Let $G$ be a virtually cyclic group. Then for any minimal action of $G$ on a compact Hausdorff topological space $C$ by homeomorphisms, the topological full group $[[G\acts C]]$ is amenable.
\end{thm}

\begin{thm}\label{thm:B}
Let $G$ be a finitely generated group that is not virtually cyclic. Then there exists a minimal free action of $G$ on a Cantor space $\Sigma$ by homeomorphisms such that the topological full group $[[G \acts \Sigma]]$ contains a non-abelian free group.
\end{thm}

\begin{remark} Some remarks are in order:
\begin{itemize}[nolistsep]
\item The minimality of the action is a necessary requirement: even for $\mathbf{Z}$, one can construct non-minimal actions such that the topological full group is non-amenable (see \cite{EM13}, section 3).
\item These results do not exclude the existence of a topological full group (of a minimal action of a finitely generated group) that is non-amenable but does not contain a non-abelien free group.
\item In \cite{JNdlS16} Juschenko, Nekreshevych and de la Salle give an amenability condition for topological full groups that is widely applicable. In fact, it implies many earlier results about amenability of topological full groups. To the best of the author's knowledge Theorem \ref{thm:A} is not a direct consequence of their result.
\end{itemize}
\end{remark}

\vspace{12pt}

\noindent{\bf Acknowledgements.} I would like to thank Nicolas Monod for introducing me to these questions and for helpful discussions along the way. I am also very grateful to Tam\'as Terpai for allowing me to include his proof for Proposition \ref{prop:coloring2}.

\section{Preliminaries}

Let $G$ be a group acting on a set $X$. Let $\mathrm{PW}(G\acts X)$ denote the \emph{piecewise group} of this action. The elements of the piecewise group are bijections on $X$ that are piecewise given by finitely many elements of $G$. Formally, $\varphi\in \mathrm{PW}(G\acts X)$ if $\varphi \colon X\rightarrow X$ is a bijection and there exists a natural number $n\in \mathbf{N}$, a finite partition $X=\sqcup_{i=1}^n X_i$ and group elements $g_1,\dots, g_n\in G$ such that for each $i\in [1,n]$, $\varphi(x)=g_i\cdot x$ for every $x\in X_i$.

\begin{remark}\label{rem:piecewise}
There is a useful equivalent characterization of piecewise maps. The bijection $\varphi$ from $X$ to itself is a piecewise $G$ map if and only if there exists a finite set $S\subset G$ such that $\varphi(x)\in S\cdot x$ for every $x\in X$. From this description it immediately follows that the piecewise $G$ maps form a group.
\end{remark}

Note that if $G$ acts on a compact topological space by homeomorphisms, then the topological full group of this action is a subgroup of the piecewise group. This is due to the compactness of the space, since a partition into open subsets is necessarily finite.

Recall that if $(P)$ is a property of groups, then we call the group $G$ \emph{locally $(P)$} if all finitely generated subgroups of $G$ have the property $(P)$.

A topological space is a \emph{Cantor space} if it is non-empty, compact, metrizable, totally disconnected and has no isolated points. Recall that any two Cantor spaces are homeomorphic to each other.

\begin{notation}
For integers $a,b\in \mathbf{Z}$ we will denote the closed interval of integers between $a$ and $b$ by $[a,b]$. So
\[ [a,b]:=\{k\in \mathbf{Z}  :  \min(a,b)\leq k\leq \max(a,b) \}.\]
\end{notation}

\subsection*{Virtually $\mathbf{Z}$ groups}

Let $G$ be a virtually $\mathbf{Z}$ group. Then we can find a finite index subgroup $H\leq G$ that is isomorphic to $\mathbf{Z}$. Let
\[N=\bigcap_{g\in G} g^{-1}Hg.\]
Since $H$ has finitely many conjugates in $G$, $N$ is a finite index subgroup of $H$, hence it also has finite index in $G$. Therefore, we found a finite index normal subgroup of $G$ that is isomorphic to $\mathbf{Z}$.

Thus, a virtually $\mathbf{Z}$ group is always an extension of a finite group by $\mathbf{Z}$. Consider the short exact sequence
\[\{0\}\longrightarrow \mathbf{Z}\stackrel{\iota}{\longrightarrow} G\stackrel{\pi}{\longrightarrow} Q\longrightarrow \{1\},\]
where $Q$ is a finite group. We will denote the identity element of $Q$ by $e_Q$ and the identity element of $G$ by $e_G$. Choose a section $x\mapsto g_x$ from $Q$ to $G$ such that $g_{e_Q}=e_G$, so we have that $\pi(g_x)=x$ for every $x\in Q$. This defines the cocycle $f\colon Q\times Q\rightarrow \mathbf{Z}$ by the equality
\[g_xg_y=\iota(f(x,y)) g_{xy} \quad \text{for } x,y\in Q.\]
The map $\alpha \colon Q\rightarrow \mathrm{Aut}(\mathbf{Z})=\{\pm1\}$, $x\mapsto \alpha_x$ defined by
\[g_x \iota(n) g_x^{-1}=\iota(n^{\alpha_x}) \quad \text{for } x\in Q \text{ and } n\in \mathbf{Z}\]
is a homomorphism. The maps $f$ and $\alpha$ determine the extension.

\begin{definition}\label{def:virtZ}
We think of a virtually $\mathbf{Z}$ group $G$ as the set $\mathbf{Z}\times Q$ with the multiplication
\[(a,x) (b,y)=(f(x,y)+a+b^{\alpha_x}, xy) \quad \text{for } a,b\in \mathbf{Z}, x,y\in Q,\]
where $Q$ is a finite group, $f\colon Q\times Q\rightarrow \mathbf{Z}$ a cocycle and $\alpha\colon  Q\rightarrow \mathrm{Aut}(\mathbf{Z})=\{\pm 1\}$ is a homomorphism.
\end{definition}

\begin{remark}
The identity element of $G$ is $e_G=(0,e_Q)$. Also note that the choice $g_{e_Q}=e_G$ implies that $f(e_Q,x)=f(x,e_Q)=0$ for all $x\in Q$.
\end{remark}

\begin{lemma}\label{le:virtZsubgr}
Let $H$ be a subgroup of the virtually $\mathbf{Z}$ group $G$. Then one of the following two cases holds.
\begin{itemize}
\item $H$ is finite and it has infinite index in $G$,
\item $H$ is infinite and it has finite index in $G$.
\end{itemize}
\end{lemma}

\begin{proof}
If $H$ is finite, then it must have infinite index, since $G$ itself is infinite. Now assume that $H$ is infinite. Using the notations in Definition \ref{def:virtZ}, there exists at least one $x\in Q$ such that $H\cap (\mathbf{Z}\times \{x\})$ is infinite. Suppose that $(a,x), (b,x)\in H$ for $a\neq b$, then $(a,x)^{-1}(b,x)\in H\cap (\mathbf{Z}\times \{e_Q\})$. Since $a\neq b$, this is not the identity element. Therefore $H\cap (\mathbf{Z}\times \{e_Q\})$ is a nontrivial subgroup, say it is equal to $k\mathbf{Z}\times \{e_Q\}$. Hence $|G:H|\leq |G:(k\mathbf{Z}\times \{e_Q\})| =k|Q|$, so the index of $H$ is finite.
\end{proof}

\section{Extensive amenability}

Consider a group $G$ acting on a set $X$. Recall that this action is called amenable if there exists a $G$-invariant mean on $X$. In this section we focus on a stronger property, namely the extensive amenability of the action.

\begin{notation}
If $X$ is a set, then the set of all finite subsets of $X$ becomes an abelian group with the symmetric difference. Let us denote this group by $\mathscr{P}_f(X)$. Note that a $G$-action on $X$ gives rise to an action on $\mathscr{P}_f(X)$.
\end{notation}

\begin{definition}
The action $G\acts X$ is \emph{extensively amenable} if there exists a $G$-invariant mean on $\mathscr{P}_f(X)$ giving full weight to the collection of sets containing any given element of $\mathscr{P}_f(X)$.
\end{definition}

Extensive amenability was introduced and first used (without a name) in \cite{JM13}. The above definition was given in \cite{JMBMdlS18}, and we can find  equivalent characterizations in \cite{JM13}, \cite{JNdlS16} and \cite{JMBMdlS18}. We will use the following two statements about extensive amenability.

\begin{lemma}[Lemma 2.2 in \cite{JMBMdlS18}]\label{le:extamen_fingen}
Let $G$ be a group acting on a set $X$. Then the following statements are equivalent.
\begin{enumerate}[nolistsep, label=(\roman*)]
\item The action of $G$ on $X$ is extensively amenable.
\item For every finitely generated subgroup $H\leq G$ and every $H$-orbit $Y\subseteq X$, the action of $H$ on $Y$ is extensively amenable.
\end{enumerate}
\end{lemma}

\begin{theorem}[Theorem 1.2 in \cite{JNdlS16}]\label{thm:recurrent}
Let $G$ be a finitely generated group with generating set $S$, acting transitively on a set $X$. If the graph of the action of $G$ on $X$ with generating set $S$ is recurrent, then the action $G\acts X$ is extensively amenable.
\end{theorem}

Recall that we call a graph recurrent if the simple random walk on this graph is recurrent. (Since we are dealing with connected vertex transitive graphs, recurrence does not depend on the starting point.) We refer the reader to  \cite{LyPer16} for an introduction to random walks. For a generating set $S\subset G$, the graph of the action $G\acts X$ is given by the vertex set $X$ and the set of edges $\{(x,sx) : x\in X, s\in S\}$. The recurrence of this graph does not depend on the generating set.

We will use the `if' direction of the following theorem.

\begin{theorem}[Varopoulos, \cite{Var86}]\label{prop:virtZrec}
A finitely generated group is recurrent if and only if it is finite, or virtually $\mathbf{Z}$, or virtually $\mathbf{Z}^2$.
\end{theorem}

In the following proposition we give a sufficient condition for the action of the piecewise group to be extensively amenable. The same statement follows from the first part of Theorem 1.4 in \cite{JdlS15} in the case of finitely generated groups.

\begin{proposition}\label{prop:extamen_fullgr}
Let $G$ be a group acting on a set $X$. Assume that for all finitely generated subgroups $H\leq G$ and all $H$-orbits $Y\subseteq X$ the graph of the action $H\acts Y$ is recurrent. Then the action of the piecewise group $\mathrm{PW}(G\acts X)$ on $X$ is extensively amenable.
\end{proposition}

\begin{remark}
By Theorem \ref{thm:recurrent} and Lemma \ref{le:extamen_fingen}, it immediately follows that the action of $G$ on $X$ is extensively amenable. Now we prove the extensive amenability of the action of the piecewise group by verifying the recurrence assumption for all finitely generated subgroups of the piecewise group.
\end{remark}

\begin{proof}
According to Lemma \ref{le:extamen_fingen}, it is sufficient to show that for a finitely generated subgroup $F\leq \mathrm{PW}(G\acts X)$, the action of $F$ on $X$ is extensively amenable. So let $F$ be any finitely generated subgroup of $\mathrm{PW}(G\acts X)$. For any $\varphi\in F$, there exists a finite set of group elements $S_{\varphi}=\{g_1,\dots, g_k\}\subset G$ such that $\varphi(x)\in S_{\varphi}\cdot x$ for all $x\in X$. Hence there exists a finitely generated subgroup $H\leq G$ such that $F\leq \mathrm{PW}(H\acts X)$.

Let $p\in X$ be an arbitrary point and let  $Y\subseteq X$ denote the $H$-orbit of $p$. By Lemma \ref{le:extamen_fingen}, it is enough to show that the $F$-action on $Y$ is extensively amenable.

Let $\mathcal{G}_F$ be the graph of the action of $F$ on $Y$, i.e., $\mathrm{V}(\mathcal{G}_F)=Y$ and $\mathrm{E}(\mathcal{G}_F)=\{(y, \varphi(y)): y\in Y,\ \varphi\in T\}$ where $T=T^{-1}\subseteq F$ is a finite generating set of $F$. This graph might not be connected.

Now each $\varphi\in T$ is a piecewise $H$ map on $X$, so we can find a finite set $S_{\varphi}\subset H$ such that $\varphi(x)\in S_{\varphi}\cdot x$ for all $x\in Y$. Let $\hat{T}=\bigcup_{\varphi\in T} S_{\varphi}\subset H$ and let $S$ be a symmetric generating set of $H$ such that $\hat{T}\subseteq S$. Let $\mathcal{G}_{H}$ denote the graph of the action of $H$ on $Y$ with generating set $S$.

The vertex set of $\mathcal{G}_F$ is equal to the vertex set of $\mathcal{G}_{H}$. Whenever $(y,\varphi(y))$ is an edge in $\mathcal{G}_F$, there exists $g\in \hat{T}$ such that $\varphi(y)=g\cdot y$. Since $\hat{T}\subseteq S$, this implies that $(y,\varphi(y))$ is also an edge of $\mathcal{G}_{H}$. Hence $\mathcal{G}_F$ is a subgraph of $\mathcal{G}_{H}$.

By the assumption in the statement the graph $\mathcal{G}_{H}$ is recurrent. As a corollary of Rayleigh's monotonicity principle, all connected subgraphs of a recurrent graph are also recurrent (for a proof see \cite{LyPer16}, Chapter 2). In particular, all connected components of $\mathcal{G}_F$ are recurrent. Hence by Theorem \ref{thm:recurrent} and Lemma \ref{le:extamen_fingen} the $F$-action on $Y$ is extensively amenable.
\end{proof}

\begin{corollary}\label{cor:extamen_locvirtZ}
Let $G$ be a locally virtually cyclic group acting on a set $X$. Then the action of the piecewise group $\mathrm{PW}(G\acts X)$ on $X$ is extensively amenable.
\end{corollary}

\begin{remark}
The same statement holds if every finitely generated subgroup of $G$ is either virtually cyclic or virtually $\mathbf{Z}^2$. Also note that we make no assumptions about the set $X$ or the action of $G$ on $X$.
\end{remark}

\begin{proof}
We use Proposition \ref{prop:extamen_fullgr} for $G$. Consider a finitely generated subgroup $H\leq G$, we know that $H$ is virtually cyclic. There are two possibilites for $H$, either it is a finite group or it is virtually $\mathbf{Z}$. In both cases, the group $H$ is recurrent, so for every action of $H$, the graph of the action is also recurrent. (This is another corollary of Rayleigh's monotonicity principe, for a proof see \cite{LyPer16}, Chapter 2.) This verifies the assumption of Proposition \ref{prop:extamen_fullgr}, so the action of $\mathrm{PW}(G\acts X)$ on $X$ is extensively amenable.
\end{proof}

\section{Proof of Theorem \ref{thm:A}}

The following is an immediate consequence of Corollary 1.4 from \cite{JMBMdlS18}.

\begin{proposition}\label{prop:subgramen}
Let $G\acts X$ be an extensively amenable action. A subgroup $\Delta$ of $\mathscr{P}_f(X)\rtimes G$ is amenable if $\Delta\cap (\{\varnothing\}\times G)$ is so.
\end{proposition}

\begin{remark}\label{rem:cocycle}
As mentioned in Remark 1.5 in \cite{JMBMdlS18}, a special case of the above proposition (already used in \cite{JM13}) is when we can construct an embedding $G\hookrightarrow \mathscr{P}_f(X)\rtimes G$, $g\mapsto (c_g,g)$ such that the subgroup $\{g\in G :  c_g=\varnothing\}\leq G$ is amenable. We call the map $G\rightarrow \mathscr{P}_f(X)$, $g\mapsto c_g$ a \emph{$\mathscr{P}_f(X)$-cocycle with amenable kernel}. The existence of such a cocycle implies the amenability of $G$ by Proposition \ref{prop:subgramen}.
\end{remark}

In this section we will consider a virtually $\mathbf{Z}$ group $G$ acting minimally on a compact Hausdorff topological space $C$ by homeomorphisms. Our goal is to prove the amenability of the topological full group $[[G\acts C]]$ using the above method, i.e., we would like to construct a cocycle on $[[G\acts C]]$ with an amenable kernel.

We will always think of $G$ as the set $\mathbf{Z}\times Q$ with the multiplication
\[(a,x) (b,y)=(f(x,y)+a+b^{\alpha_x}, xy) \text{ for } a,b\in \mathbf{Z}, x,y\in Q,\]
as described in Definition \ref{def:virtZ}.

\subsection{Action on an orbit}

Take a point $p$ in the space $C$. Since the action of $G$ is minimal, the orbit of $p$ is dense. This means in particular that the stabilizer $G_p$ has infinite index in the virtually $\mathbf{Z}$ group $G$, so it is finite by Lemma \ref{le:virtZsubgr}. Let us denote the orbit of $p$ by $X$, consider the action $G\acts X$ that is the restriction of the action on $C$. Let us define the map
\begin{align*}
\varepsilon_p\colon [[G\acts C]] & \hooklongrightarrow  \mathrm{PW}(G\acts X),\\
\varphi & \longmapsto  \varphi_{\big|X}.
\end{align*}
Since $X$ is dense in $C$, the action of $\varphi$ on $X$ determines $\varphi$ on $C$, so this map is injective. 

In many cases it will be more convenient to work with the piecewise group $\mathrm{PW}(G\acts X)$ instead of the topological full group $[[G\acts C]]$.

\vspace{4mm}

We will use the following technical lemma and remark in the construction of our cocycle.

\begin{lemma}\label{le:stabilizer}
If $G$ is a virtually $\mathbf{Z}$ group acting on a compact space $C$ minimally, then we can find a point $p\in C$ such that the normalizer $N_G(G_p)$ is infinite.
\end{lemma}

\begin{proof}
Suppose that for all points $q\in C$, the normalizer $N_G(G_q)$ is finite. The stabilizer of $g\cdot q$ is $gG_qg^{-1}$, so the finiteness of all normalizers implies that for every $q$, there are only finitely many points in the orbit of $q$ that have the same stabilizer as $q$, and all the others have conjugate stabilizers.

Consider the map from $\mathrm{Sub}_f(G)=\{K\leq G : K\text{ is finite}\}$ to $\mathscr{C}(C)=\{C'\subseteq C :  C' \text{ is closed}\}$ that gives us the set of points stabilized by a certain subgroup. For $K\leq G$, let
\[\xi(K)=\{q\in C :  g\cdot q=q \text{ for every } g\in K\}. \]
This is indeed a closed set in $C$, since $G$ acts by homeomorphisms. Note that $q\in \xi(K)$ implies $K\leq G_q$. We have 
\[C=\bigcup_{K=G_q \text{ for some } q\in C} \xi(K).\]
Since $\mathrm{Sub}_f(G)$ itself is countable, we can have at most countably many subgroups as stabilizers. So $C$ is the union of countably many closed sets. The Baire category theorem implies that at least one of these sets must have non-empty interior. Let $K_0\leq G$ such that $\mathrm{int}(\xi(K_0))\neq \varnothing$.

Pick a point $q\in C$ such that $K_0=G_q$. As we saw in the beginning of the proof, there are only finitely many points in the orbit of $q$ that are stabilized by $K_0$. Hence this orbit intersects the non-empty open set $\mathrm{int}(\xi(K_0))$ at finitely many points, so it cannot be dense in $C$. This contradicts the minimality of the action.
\end{proof}

\begin{remark}\label{rem:stabilizer}
By Lemma \ref{le:stabilizer}, we can always find a point $p$ such that $N_G(G_p)$ is infinite, so it has finite index in the virtually $\mathbf{Z}$ group $G$ by Lemma \ref{le:virtZsubgr}. Hence its intersection with $\mathbf{Z}\times \{e_Q\}$ is $k\mathbf{Z}\times \{e_Q\}$ for some $k\in \mathbf{N}$. We also know that $G_p$ intersects $\mathbf{Z}\times \{e_Q\}$ trivially, since $G_p$ is finite. Hence $G_p$ and $k\mathbf{Z}\times \{e_Q\}$ normalize each other and their intersection is trivial, so they commute. Therefore, we can assume (perhaps by passing to a finite index subgroup of $\mathbf{Z}$) that $G_p$ commutes with the normal subgroup $\mathbf{Z}\times \{e_Q\}$ in $G$. We will use this assumption later.
\end{remark}

\begin{definition}[Generators]\label{def:genset}
As before, let $G$ be a virtually $\mathbf{Z}$ group acting transitively on the space $X$, and let $p\in X$ be a point and $H=G_p$ its stabilizer. Using the notations of Definition \ref{def:virtZ}, $H$ intersects $(\mathbf{Z}\times \{e_Q\})\leq G$ trivially. With $\pi$ denoting the projection map $G\rightarrow Q$, the subgroup $\pi(H)$ is isomorphic to $H$. Let $m$ be the index of $\pi(H)$ in $Q$, and let $\{x_1,\dots, x_m\}$ be a system of coset representatives of $\pi(H)$ in $Q$ such that $x_1=e_Q$. Let us introduce the notations
\begin{align*}
\tau & =(1,e_Q)\in G,\\
\sigma_i & =(0,x_i)\in G \text{ for } i\in[1,m].
\end{align*}
The choice of the generating set $S=\{\tau,\sigma_2,\dots, \sigma_m\}$ ensures that the Schreier graph $\mathrm{Sch}(G,H,S)$ connected. In other words the graph of the action of $G$ on $X$ is connected with generating set $S$. (Since $\sigma_1=e_G$ it is not necessary to have this element among the generators, but this notation will be convenient in some calculations.)
\end{definition}

\begin{lemma}\label{le:genset}
We have $X=G\cdot p=\{\tau^k \sigma_i  \cdot p=(k,x_i)\cdot p  : k\in \mathbf{Z}, i\in [1, m]\}$. Furthermore $\tau^k \sigma_i \cdot p=\tau^\ell \sigma_j \cdot p$ if and only if $(k,i)=(\ell,j)$.
\end{lemma}

\begin{proof}
Observe that
\begin{align*}
\tau \sigma_i & =(1,e_Q) (0,x_i)=(1,x_i),\\
\sigma_i \tau & =(0,x_i) (1,e_Q)=(1^{\alpha_{x_i}},x_i)=\tau^{\pm 1} \sigma_i.
\end{align*}
Hence $G\cdot p=\{\tau^k \sigma_i  \cdot p=(k,x_i)\cdot p  : k\in \mathbf{Z}, i\in [1, m]\}$. If  $\tau^k \sigma_i \cdot p=\tau^\ell \sigma_j \cdot p$, then $\sigma_j^{-1}\tau^{k-\ell}\sigma_i=\tau^{\pm (k-\ell)}\sigma_j^{-1}\sigma_i\in H$. By taking the $\pi$-image we have 
\[\pi(\tau^{\pm (k- \ell)}\sigma_j^{-1} \sigma_i)=x_j^{-1}x_i\in \pi(H).\]
Since these are coset representatives of $\pi(H)$, we have $x_i=x_j$ and $i=j$. Hence $\sigma_i=\sigma_j$, and $\tau^{k-\ell}\in H$ implies $k=\ell$. 

Therefore, $\tau^k \sigma_i \cdot p=\tau^\ell \sigma_j \cdot p$ if and only if $(k,i)=(\ell,j)$.
\end{proof}

\subsection{Definition of the cocycle}

\begin{definition}\label{def:cocycle}
Let $G$ be a virtually $\mathbf{Z}$ group acting minimally on the compact space $C$, and let $p\in C$ with the assumption of Remark \ref{rem:stabilizer}. Let $X=G\cdot p$ as before. Let us define $A\subseteq X$ as follows.
\[A=\{ (k,x_i)\cdot p  :  k\in\mathbf{Z}, i\in [1,m], k^{\alpha_{x_i}}\geq 0\}\]
\end{definition}

Observe that $A$ is a collection of `half lines'. Whenever $\alpha_{x_i}$ is the identity, $A$ contains the positive half line $\{(k,x_i)\cdot p  : k\in \mathbf{N}\}$, and when $\alpha_{x_i}=-1$, it contains the negative half line $\{(k,x_i)\cdot p  :  -k\in \mathbf{N}\}$.

\begin{lemma}\label{le:cocycle1}
For every $g\in G$, the set $gA \setminus A\subset X$ is finite.
\end{lemma}

\begin{proof}
Let $g=(\ell,y)$ be an arbitrary element of $G$. We have
\[g\cdot ((k,x_i)\cdot p)=(\ell,y)(k,x_i)\cdot p=(f(y,x_i)+\ell+k^{\alpha_y}, yx_i)\cdot p.\]
There exists a unique $j\in [1,m]$, such that $yx_i\pi(H)=x_j\pi(H)$ and a unique $(t,z)=h\in H$, such that $yx_iz=x_j$. So we have
\[(f(y,x_i)+\ell+k^{\alpha_y}, yx_i)\cdot p=(f(y,x_i)+\ell+k^{\alpha_y}, yx_i)(t,z)\cdot p=\]
\[=(f(yx_i,z)+f(y,x_i)+\ell+k^{\alpha_y}+t^{\alpha_{yx_i}},x_j)\cdot p.\]
We are going to use the assumption that $H$ commutes with $\mathbf{Z}\times \{e_Q\}$, mentioned in Remark \ref{rem:stabilizer}. This implies that $\alpha(\pi(H))=\{+1\}$, so whenever $x$ and $x'$ are in the same coset of $\pi(H)$ in $Q$, we have $\alpha_x=\alpha_{x'}$. Now we know that $yx_i$ and $x_j$ are in the same coset, so $\alpha_{yx_i}=\alpha_{x_j}$. Therefore, if the element above is not inside $A$, then \[0>(f(yx_i,z)+f(y,x_i)+\ell+k^{\alpha_y}+ t^{\alpha_{yx_i}})^{\alpha_{yx_i}}= f(yx_i,z)^{\alpha_{yx_i}}+f(y,x_i)^{\alpha_{yx_i}}+ \ell^{\alpha_{yx_i}} +k^{\alpha_{x_i}}+t.\]
For a fixed $g$ and fixed $i\in [1,m]$, the number $f(yx_i,z)^{\alpha_{yx_i}}+f(y,x_i)^{\alpha_{yx_i}}+ \ell^{\alpha_{yx_i}} +t$ is also fixed. For $i\in [1,m]$ let $A_i=\{ (k,x_i)\cdot p :  k\in \mathbf{Z}, k^{\alpha_{x_i}}\geq 0\}$. The previous calculations show that
\[ |gA_i\setminus A | =|\{k\in \mathbf{Z} :  k^{\alpha_{x_i}}\geq 0 \text{ and } f(yx_i,z)^{\alpha_{yx_i}}+ f(y,x_i)^{\alpha_{yx_i}}+\ell^{\alpha_{yx_i}} +k^{\alpha_{x_i}}+t < 0 \}|, \]
which is finite. Hence,
\[gA\setminus A=\bigcup_{i\in [1,m]} (gA_i\setminus A)\]
is also a finite set.
\end{proof}

\begin{proposition}\label{prop:cocycle2}
For every piecewise map $\varphi\in \mathrm{PW}(G\acts X)$, the set $A\triangle \varphi(A)$ is finite.
\end{proposition}

\begin{proof}
There exists a finite set $T=\{g_1,\dots, g_t\}\subset G$, such that for every $h\in G$, $\varphi(h\cdot p)\in Th\cdot p$. We have the following inclusion.
\[\varphi(A)\setminus A\subseteq \left(\bigcup_{i=1}^t g_iA \right) \setminus A=\bigcup_{i=1}^t (g_iA\setminus A).\]
By Lemma \ref{le:cocycle1}, the sets on the right-hand side are all finite, so $\varphi(A)\setminus A$ is also finite. The same argument works for $\varphi^{-1}(A)\setminus A$, implying that $A\setminus \varphi(A)$ is also finite. Hence $A\triangle\varphi(A)$ is finite.
\end{proof}

Proposition \ref{prop:cocycle2} allows us to define the following embedding.

\begin{definition} For $\varphi\in \mathrm{PW}(G\acts X)$ let $c_{\varphi}=A\triangle \varphi(A)\in \mathscr{P}_f(X)$. The map
\begin{align*}
 [[G\acts C]] \leq \mathrm{PW}(G\acts X) & \hooklongrightarrow \mathscr{P}_f(X)\rtimes \mathrm{PW}(G\acts X),\\
\varphi & \longmapsto (A\triangle \varphi(A), \varphi)
\end{align*}
is an embedding.
\end{definition}

\subsection{Amenable kernel}

This section is dedicated to proving that the kernel of the cocycle $[[G\acts C]]\rightarrow \mathscr{P}_f(X)$, $\varphi\mapsto c_{\varphi}=A\triangle \varphi(A)$ defined in the previous section is amenable. The kernel is the subgroup $\{\varphi\in [[G\acts C]] :  \varphi(A)=A\}$, i.e., the stabilizer of the set $A\subseteq X$ in the full group $[[G\acts C]]$.

\begin{remark}\label{rem:normalmin}
Let $G$ be a group acting on a compact space $C$ minimally. If $N\triangleleft G$ is a finite index normal subgroup, then there exists a minimal $N$-invariant closed subset $C_0\subseteq C$. Note that for any $g\in G$, the set $gC_0$ is also a minimal $N$-invariant closed set, hence it is equal to $C_0$ or they are disjoint. So for any two elements $g_1,g_2\in G$, the sets $g_1C_0$ and $g_2C_0$ are either equal or disjoint. Since $N$ has finite index in $G$, there are finitely many elements $\gamma_1,\dots, \gamma_n\in G$ such that 
\[C=\bigsqcup_{i=1}^n \gamma_i C_0.\]
\end{remark}

We leave the proof of the following lemma to the reader.

\begin{lemma}\label{le:invariantsets}
Let $G$ be a group acting on a topological space $X$ by homeomorphisms. Assume that we can divide the space to finitely many minimal closed $G$-invariant subsets, say $X=\bigsqcup_{i=1}^n X_i$. Now let $Y\subseteq X$ be an open or a closed $G$-invariant set. Then there exists $I\subseteq [1,n]$, such that $Y=\bigsqcup_{i\in I} X_i$.
\end{lemma}

\begin{definition}
Let $G$ be a group acting on the space $C$. Let $D\subset G$ be a finite set containing the identity element. For an element $\varphi\in \mathrm{PW}(G\acts C)$ and for two points $q_1,q_2\in C$, we say that the \emph{$\varphi$-action is the same on the $D$-neighborhood of $q_1$ and $q_2$}, if for every $d\in D$, $\varphi$ acts by the same element of $G$ on $d\cdot q_1$ and on $d\cdot q_2$ (i.e., there exists $g\in G$ such that $\varphi(d\cdot q_1)=gd\cdot q_1$ and $\varphi(d\cdot q_2)=gd\cdot q_2$).
\end{definition}

\begin{lemma}\label{le:upp}
Let $G$ be a virtually $\mathbf{Z}$ group acting minimally on the compact space $C$, and let $p\in C$ be as before.

For every finite subset $F\subset [[G\acts C]]$ and every $n\in \mathbf{N}$, there exists $q=q(n,F)\in \mathbf{N}$, such that for every interval $I\subset \mathbf{Z}$ of length $2q$, there exists $t\in I$ such that $[t-n,t+n]\subseteq I$, and for every $\varphi\in F$, the $\varphi$-action is the same on the $[-n,n]\times Q$-neighborhood of $p$ and $(-t,e_Q)\cdot p$.
\end{lemma}

\begin{proof}
Let $F\subset [[G\acts C]]$ be a finite subset and $n\in \mathbf{N}$. We can find a finite partition $\mathcal{P}$ of $C$ to clopen (closed and open) sets, such that every $\varphi\in F$ is just acting with an element of $G$ when restricted to an element of $\mathcal{P}$. This implies that there exists an open neighborhood $V$ of $p$, such that for all $j\in [-n,n]$ and all $x\in Q$, the set $(j,x)\cdot V$ is contained in some $P\in \mathcal{P}$. The union
\[\bigcup_{\ell\geq 1}\bigcup_{|r|\leq \ell} (r,e_Q)\cdot V\]
is nonempty, open and $Z$-invariant, where $Z=\mathbf{Z}\times \{e_Q\}\triangleleft G$.
By Remark \ref{rem:normalmin}, there exists a minimal $Z$-invariant closed subset $C_0\subseteq C$ with $p\in C_0$, a natural number $M\in \mathbf{N}$, and group elements $e_G=g_1, g_2,\dots, g_M\in G$, such that 
\[C=\bigsqcup_{i=1}^M g_i\cdot C_0.\]

By Lemma \ref{le:invariantsets}, the nonempty, open and $Z$-invariant set $\bigcup_{\ell\geq 1}\bigcup_{|r|\leq \ell}(r,e_Q)\cdot V$ is the union of some $g_i\cdot C_0$'s. Since it contains $p$, it has a nonempty intersection with $C_0$, hence we have
\[C_0\subseteq \bigcup_{\ell\geq 1}\bigcup_{|r|\leq \ell} (r,e_Q)\cdot V.\]
By the compactness of $C_0$, a finite union already contains it, so there is $\ell\in \mathbf{N}$ such that 
\[C_0\subseteq \bigcup_{|r|\leq \ell} (r,e_Q)\cdot V.\]
Let $q=q(n,F)=\ell+n$. Now if $I\subset \mathbf{Z}$ is an interval of length $2q$, then $I=[s-q,s+q]$ for some $s\in \mathbf{Z}$. We have 
\[C_0=(s,e_Q)\cdot C_0\subseteq \bigcup_{|r|\leq \ell} (s,e_Q)(r,e_Q)\cdot V=\bigcup_{|r|\leq \ell} (s+r,e_Q)\cdot V.\]
Hence there is an integer $t\in [s-\ell,s+\ell]$, such that $p\in (t,e_Q)\cdot V$. By the choice of $q$, we have $[t-n,t+n]\subseteq [s-q,s+q]$. By multiplying with the inverse of $(t,e_Q)$, we get
\[(t,e_Q)^{-1}\cdot p=(-t,e_Q)\cdot p\in V.\]
Thus for all $j\in [-n,n]$ and for all $x\in Q$, both $(j,x)\cdot p$ and $(j,x)(-t,e_Q)\cdot p$ are in $(j,x)\cdot V$, so every element of $F$ acts on them as the same element of $G$ (by the choice of the partition and the open set $V$). In other words, the action of every element of $F$ is the same on the $[-n,n]\times Q$-neighborhoods of $p$ and $(-t,e_Q)\cdot p$.
\end{proof}

The picture we have in mind is the following. We think of $X=G\cdot p$ as a collection of lines, for each $i\in [1,m]$ the corresponding line is $(\mathbf{Z}\times \{x_i\})\cdot p$, and the origin on this line is the point $(0,x_i)\cdot p$. Suppose that we only know the action of $F$ on a neighborhood of the origin on each line (i.e., the $([-n,n]\times \{x_1,\dots, x_m\})$-neighborhood of $p$). Lemma \ref{le:upp} states that on any chosen line in any long enough interval we can find a subinterval of length $2n$, where the action of $F$ is determined by what we know.

Having these lines in mind we introduce a definition to get the coordinate of a given point on the corresponding line.

\begin{definition}
For $q\in X=G\cdot p$, there exists exactly one element of the form $(k,x_i)\in G$, such that $q=(k,x_i)\cdot p$. Let $|q|=|k|$.
\end{definition}

The following lemma shows that there is a universal bound on the distance between the coordinate of a point in $X$ and that of its image under the action of a piecewise $G$ map.

\begin{lemma}\label{le:coorddiff}
Let $\varphi\in \mathrm{PW}(G\acts X)$. Then 
\[ \sup\{ \big| |q|-|\varphi(q)| \big|  :  q\in X\}\]
is finite.
\end{lemma}

\begin{proof}
Consider the graph of the action $G\acts X$ with the generating set $S$ introduced in Definition \ref{def:genset} (this assures that the graph is connected). Let $\mathrm{d}$ denote the distance function on this graph. The piecewise group $\mathrm{PW}(G\acts X)$ also acts on the graph, take $\varphi\in \mathrm{PW}(G\acts X)$. Since this is a piecewise $G$ map, it can only move points to a limited distance. Let
\[|\varphi |=\max\{\mathrm{d}(q,\varphi(q)) : q\in X\}.\]

This means that we can get to $\varphi(q)$ from $q$ by applying at most $|\varphi |$ many generators. When we multiply a point with the generator $\tau=(1,e)$, the coordinate changes by 1. Now take $\sigma_i=(0,x_i)\in S$. We would like to compute the difference of the coordinates of $q$ and $\sigma_i\cdot q$. Let $q=(n,x_j)\cdot p$, then we have
\[\sigma_i\cdot q=(0,x_i)(n,x_j)\cdot p= (f(x_i,x_j)+n^{\alpha_{x_i}},x_ix_j)\cdot p.\]
There exists a unique $h=(\ell,z)\in H=G_p$ such that $x_ix_jz=x_k$ (for some $k\in [1,m]$).
\[\sigma_i\cdot q= (f(x_i,x_j)+n^{\alpha_{x_i}},x_ix_j)(\ell,z)\cdot p=\]
\[=(f(x_ix_j,z)+f(x_i,x_j)+n^{\alpha_{x_i}}+ \ell^{\alpha_{x_ix_j}}, x_k)\cdot p.\]
The difference of the coordinates is
\[ \big| |q|-|\sigma_i(q)| \big|=\big| | f(x_ix_j,z)+f(x_i,x_j)+n^{\alpha_{x_i}}+ \ell^{\alpha_{x_ix_j}} |- |n| \big| \leq \]
\[\leq | f[x_ix_j,z)|+ | f(x_i,x_j)| + |\ell |.\]
Recall that $f:Q\times Q\to \mathbf{Z}$ is a cocycle. Since $Q$ is finite, there is an upper bound on the absolute values $f$ can take, let us denote this number by $f_0\in \mathbf{Z}$. We also know that $H$ is finite, so there are only finitely many possible values for the first coordinate of an element of $H$, let us denote the maximum absolute value by $\ell_0\in \mathbf{Z}$. We obtained that
\[ \big| |q|-|\sigma_i(q)| \big| \leq 2f_0+\ell_0,\]
and this bound does not depend on the choice of the point $q$. So we have
\[ \big| |q|-|\varphi(q)| \big| \leq |\varphi | (2f_0+\ell_0)\]
for every $q\in X$. This finishes the proof of the lemma.
\end{proof}

\begin{notation}\label{not:coorddiff}
For a finite set $F\subset \mathrm{PW}(G\acts X)$, let us introduce the following notation.
\[d_F=\max\{ \big| |q|-|\varphi(q)| \big|  :  \varphi\in F, \ q\in X\}.\]
By Lemma \ref{le:coorddiff}, we know that this is finite.
\end{notation}

\begin{proposition}\label{prop:amenkernel}
The stabilizer $[[G\acts C]]_{A}$ is locally finite.
\end{proposition}

\begin{proof}
Let $F\subset [[G\acts C]]_{A}$ be a finite subset of the stabilizer. We would like to prove that the subgroup generated by $F$ in $[[G\acts X]]_{A}$ is finite.

By Lemma \ref{le:upp} for $F$ and $n=d_F+1$, we have $q=q(n,F)$. Let $I_0=[-q,q]$, and decompose $\mathbf{Z}$ as the disjoint union of consecutive intervals $\{I_k\}_{k\in \mathbf{Z}}$ of length $2q$. (So for example $I_1=[q+1,3q+1]$, $I_{-1}=[3q-1,q-1]$, etc.) According to the lemma, for each $k\in \mathbf{Z}\setminus \{0\}$, there exists $s_k\in I_k$ such that $[s_k-n,s_k+n]\subseteq I_k$, and for every $\varphi\in F$, the $\varphi$-action is the same on the $[-n,n]\times Q$-neighborhood of $(-s_k,e_Q)\cdot p$ and of $p$. Set $t_k=-s_{-k}$ for $k\in \mathbf{Z}\setminus \{0\}$ and $t_0=0$. This way $[t_k-n,t_k+n]\subseteq -I_{-k}=I_k$, and the $\varphi$-action is the same on the $[-n,n]\times Q$-neighborhood of $(t_k,e_Q)\cdot p$ and of $p$, for every $k\in \mathbf{Z}$.

Now for $k\in \mathbf{Z}$ and $i\in [1,m]$, let
\[B_{k,i}=\{(\ell,x_i)\cdot p :  \ell\in [t_k^{\alpha_{x_i}},(t_{k+1}-1)^{\alpha_{x_i}}] \}=\]
\[= \left( [t_k^{\alpha_{x_i}},(t_{k+1}-1)^{\alpha_{x_i}}] \times \{x_i\} \right) \cdot p \subset (\mathbf{Z}\times \{x_i\})\cdot p. \]
Note that for $\alpha_{x_i}=-1$ this interval becomes $[-t_{k+1}+1,-t_k]$. For $k\in \mathbf{Z}$ we define
\[B_k=\bigcup_{i=1}^m B_{k,i}\subset X.\]

\begin{claim}
For every $k\in \mathbf{Z}$, the finite set $B_k\subset X$ is invariant under the action of $F$.
\end{claim}

\begin{proof}
Let us fix $k\in \mathbf{Z}$. Take arbitrary elements $\varphi\in F$ and $(r,x_i)\cdot p \in B_k$, then there exists a group element $g=(\ell,z)\in G$, such that the action of $\varphi$ on $(r,x_i)\cdot p$ is multiplication by $g$. So we have \[\varphi((r,x_i)\cdot p)=(\ell,z)(r,x_i)\cdot p=(f(z,x_i)+\ell+r^{\alpha_z},zx_i)\cdot p.\]
There is a unique $(L,y)\in H$, such that $zx_iy=x_j$ for some $j\in [1,m]$.
\[\varphi((r,x_i)\cdot p)=(f(z,x_i)+\ell+r^{\alpha_z},zx_i)(L,y)\cdot p=\]
\[=(f(z,x_i)+\ell+r^{\alpha_z}+L^{\alpha_{zx_i}} +f(zx_i,y),x_j)\cdot p.\]
Let $R=f(z,x_i)+\ell+L^{\alpha_{zx_i}}+f(zx_i,y)$, by Lemma \ref{le:coorddiff} we have $|R|\leq d_F=n-1$, since this is the difference of the coordinates of $(r,x_i)\cdot p$ and its $\varphi$-image.
There are three cases.

\begin{enumerate}[label={\arabic*.}]
\item $r\in [(t_k+n)^{\alpha_{x_i}},(t_{k+1}-n-1)^{\alpha_{x_i}}]$,
\item $r\in [t_k^{\alpha_{x_i}},(t_k+n-1)^{\alpha_{x_i}}]$,
\item $r\in [(t_{k+1}-n)^{\alpha_{x_i}},(t_{k+1}-1)^{\alpha_{x_i}}]$.
\end{enumerate}

In each case we conclude that the image of the point $(r,x_i)\cdot p$ stays in the set $B_k$.

\begin{enumerate}[label={\arabic*.}]

\item Note that we have $\alpha_{x_j}=\alpha_{zx_i}= \alpha_z\alpha_{x_i}$, since they are in the same coset of $\pi(H)$ in $Q$. Therefore, in this case we have
\[r^{\alpha_z}\in [(t_k+n)^{\alpha_{x_j}},(t_{k+1}-n-1)^{\alpha_{x_j}}]. \]
Moreover, we know that $|R|< n$, and hence by acting with $\varphi$ we cannot leave the interval, i.e.,
\[r^{\alpha_z}+R \in [(t_k)^{\alpha_{x_j}}, (t_{k+1}-1]^{\alpha_{x_j}} ]. \]
This means that $\varphi((r,x_i)\cdot p)\in B_k$.

\item In this case the point $(r,x_i)\cdot p$ is in the $[-n,n]\times Q$-neighborhood of $(t_k,e_Q)\cdot p$, so we can use that the $\varphi$-action on this neighborhood of $(t_k,e_Q)\cdot p$ is the same as on the $[-n,n]\times Q$-neighborhood of $p$. This, and the fact that $\varphi$ is in the stabilizer of $A$, ensures that the image of $(r,x_i)\cdot p$ stays in $B_k$.

Let $b=r-t_k^{\alpha_{x_i}}$, then $b^{\alpha_{x_i}}\in [0,n]$, so $(r,x_i)\cdot p=(b+t_k^{\alpha_{x_i}},x_i)\cdot p=(b,x_i)(t_k,e_Q)\cdot p$. Now we can see the action of $\varphi$ on $(r,x_i)\cdot p$ the following way.
\begin{align*}
\varphi((r,x_i)\cdot p) & =\varphi((b,x_i)(t_k,e_Q)\cdot p)=\\
=(\ell,z)(b,x_i)(t_k,e_Q)\cdot p & =(\ell,z)(b+t_k^{\alpha_{x_i}},x_i)\cdot p=\\
=(\ell,z)(t_k^{\alpha_{x_i}},e_Q)(b,x_i)\cdot p & =(\ell+(t_k^{\alpha_{x_i}})^{\alpha_z}, z)(b,x_i)\cdot p=\\
= (t_k^{\alpha_{x_i}\alpha_z},e_Q)(\ell,z)(b,x_i)\cdot p & =(t_k^{\alpha_{x_iz}},e_Q)\varphi((b,x_i)\cdot p).
\end{align*}
The last equality is due to the fact that $(b,x_i)\cdot p$ is in the $[-n,n]\times Q$-neighborhood of $p$, this is the corresponding point to $(r,x_i)\cdot p$, so the $\varphi$-action on this point is also multiplication by $g=(\ell,z)$. We have
\begin{align*}
\varphi((b,x_i)\cdot p) & =(t_k^{\alpha_{x_iz}},e_Q)^{-1}\varphi((r, x_i)\cdot p)=\\
=(-t_k^{\alpha_{x_iz}},e_Q)(r^{\alpha_z}+R,x_j)\cdot p & = (r^{\alpha_z}+R-t_k^{\alpha_{x_iz}},x_j) \cdot p.
\end{align*}
Since $\varphi$ is in the stabilizer of the set $A\subset X$, the image of $(b,x_i)\cdot p$ is also in $A$, so we have
\begin{align*}
0\leq ( r^{\alpha_z}+R-t_k^{\alpha_{x_iz}} \big)^{\alpha_{x_j}} & =\\
=\big( (b+t_k^{\alpha_{x_i}})^{\alpha_z}+R-t_k^{\alpha_{x_iz}} \big)^{\alpha_{x_j}} & = b^{\alpha_{x_i}}+R^{\alpha_{x_j}}
\end{align*}
In the second equation we used that from $\alpha_{x_j}=\alpha_{x_iz} =\alpha_{zx_i}=\alpha_{x_i}\alpha_z$ we get $\alpha_z\alpha_{x_j}=\alpha_{x_i}$, since $\mathrm{im}( \alpha) \subseteq \mathrm{Aut}(\mathbf{Z})=\{\pm 1\}$. We also know that $|b|+|R|\leq 2n-1$, so
\[0\leq b^{\alpha_{x_i}}+R^{\alpha_{x_j}}\leq 2n-1.\]
From the choice of the $t_k$'s it is clear that $|t_{k+1}-t_k|\geq 2n$, so we have
\[ [t_k,t_{k+1}-1] \ni t_k+b^{\alpha_{x_i}}+R^{\alpha_{x_j}}= \big( (t_k^{\alpha_{x_i}}+b)^{\alpha_z}+R \big)^{\alpha_{x_j}}= (r^{\alpha_z}+R)^{\alpha_{x_j}},\]
\[ [t_k,t_{k+1}-1]^{\alpha_{x_j}} \ni r^{\alpha_z}+R.\]
Therefore, $\varphi((r,x_i)\cdot p)\in B_k$ holds in the second case as well.

\item In this case we use that the complement of $A\subset X$ is invariant under the action of $\varphi$. Hence, exactly the same ideas and similar calculations as in case 2.~show that $\varphi((r,x_i)\cdot p)\in B_k$.
\end{enumerate}

This proves that $B_k$ is indeed invariant under the action of all elements of $F$.
\end{proof}

Since every $B_k$ is invariant under the action of $F$, we can realize the group $\langle F\rangle$ as a subgroup of $\prod_{k\in \mathbf{Z}} \mathrm{Sym}(B_k)$. By the choice of $t_k\in I_k$, we have that $|t_{k+1}-t_k|\leq 4q$, so $|B_{k,i}|\leq 4q$ for every $i\in [1,m]$, and hence $|B_k|\leq 4qm$ for all $k\in \mathbf{Z}$. This means that there is a uniform bound on the cardinality of the $B_k$'s, so the direct product $\prod_{k\in\mathbf{Z}} \mathrm{Sym}(B_k)$ is locally finite. The group $\langle F\rangle$ is a finitely generated subgroup of this, consequently it is finite.
\end{proof}

\begin{proof}[Proof of Theorem \ref{thm:A}]
The action $\mathrm{PW}(G\acts C) \acts C$ is extensively amenable by Corollary \ref{cor:extamen_locvirtZ}. As $[[G\acts C]]$ is a subgroup of the piecewise group, the action $[[G\acts C]]\acts C$ is also extensively amenable. The kernel of the cocycle $[[G\acts C]]\rightarrow \mathscr{P}_f(X)$, $\varphi \mapsto c_{\varphi}=A\triangle \varphi(A)$ is exactly the stabilizer $[[G\acts C]]_A$. By Proposition \ref{prop:amenkernel} this kernel is locally finite, hence amenable. Therefore, by Proposition \ref{prop:subgramen}, the topological full group $[[G\acts C]]$ is amenable.
\end{proof}

\section{Proof of Theorem \ref{thm:B}}

In this section $G$ will be an infinite, not virtually $\mathbf{Z}$ group. Our goal is to construct a minimal free action of $G$ on a Cantor space such that the topological full group contains a free group.

\begin{notation}
We denote the identity element of $G$ by $e$.

Let $S$ be a symmetric generating set for the group $G$. If $a$ and $b$ are subsets or elements of $G$, then we will denote their distance in the Cayley graph $\mathrm{Cay}(G,S)$ by $\mathrm{d}(a,b)$. For $g\in G$, let $\mathrm{length}(g)$ denote $\mathrm{d}(e,g)$.

Let $k\in \mathbf{N}$. For $x\in G$ let $B_k(x)$ denote the $k$-ball around $x$ in the Cayley graph. For a subset $X\subseteq G$, we will denote its $k$-neighborhood by $B_k(X)$, i.e., $B_k(X)=\bigcup_{x\in X} B_k(x)$.
\end{notation}

We will work with edge colorings of graphs. We introduce the following notion.

\begin{definition}
Let $A,B,C, D_1, D_2, \dots D_k$ denote $k+3$ different colors and let $X\in \{A,B,C,D_1, D_2, \dots , D_k\}$ be one of the colors. An edge coloring of a graph is \emph{$X$-proper}, if there are no adjacent $X$-colored edges in the graph. We will call a coloring \emph{3-proper} if it is $A$-proper, $B$-proper and $C$-proper.

Let $G$ be a group with a symmetric generating set $S$. Let $\Sigma_{G}$ denote the space of 3-proper edge colorings of the Cayley graph $\mathrm{Cay}(G,S)$ by the letters $A,B,C,D, E, F$.
\end{definition}

Note that with the topology of pointwise convergence, $\Sigma_G$ is a compact, metrizable, totally disconnected space. We will consider the natural (left) $G$-action on $\Sigma_G$ defined by translations.

For each color $X\in \{A,B,C\}$, there exists a corresponding continuous involution on $\Sigma_G$, that we will denote by the same letter. On $\sigma\in \Sigma_G$ it is defined as follows: if the vertex $e\in G$ is adjacent to an edge labeled by $X$, then we translate the coloring towards the other endpoint of this edge (i.e., the origin is now at that other vertex), if $e$ has no adjacent edges labeled by $X$ then $ X\cdot \sigma=\sigma$. (This is well-defined since the coloring is 3-proper, so we have at most one $X$-colored edge from every vertex.) This involution is contained in the topological full group $[[G\acts \Sigma_G]]$. Note that the involution preserves any $G$-invariant subset of $\Sigma_G$, so if $M$ is a closed $G$-invariant subset, then $A,B$ and $C$ can also be viewed as elements of the topological full group $[[G\acts M]]$. This gives us a homomorphism from the free product $\Delta=\langle A\rangle \ast \langle B \rangle \ast \langle C\rangle$ to $[[G\acts M]]$. 
Since $\Delta$ is isomorphic to $(\mathbf{Z}/2\mathbf{Z}) \ast (\mathbf{Z}/2\mathbf{Z}) \ast (\mathbf{Z}/2\mathbf{Z})$, it contains the free group on two generators. Our goal is to find a subspace $M$ so that this homomorphism is injective, proving that $[[G\acts M]]$ contains a non-abelian free group.

\vspace{12pt}

We will think of elements of $\Delta$ as words using the letters $A,B,C$, that do not contain two consecutive instances of the same letter. There are no powers or inverses needed, since the generators are all involutions. For $w\in \Delta$, $\mathrm{length}(w)$ denotes its length as a word from the letters $A,B,C$.

Given a coloring $\sigma\in \Sigma_G$, we will say that a word $w\in \Delta$ is \emph{written} along a path $v_0,v_1,v_2,\dots, v_k$, if the length of $w$ is $k$ and the color of the edge $(v_{i-1},v_i)$ is the $i$th letter of $w$.

\begin{definition}
Let $w$ be a word from $\Delta= \langle A\rangle \ast \langle B \rangle \ast \langle C\rangle$. Assume that in a coloring $\sigma$ the word $w$ is written along a path $v_0,v_1,\dots v_k$. We will say that in this appearance of the word $w$, the point $v_i$ is \emph{marked} with the color $X$ (with $X\in \{D,E,F\}$), if all edges going from $v_i$ have the color $X$ except for $(v_{i-1},v_i)$ and $(v_i,v_{i+1})$. (In the case when $i=0$ or $i=k$, then one of these does not exist, so there is only one exception.)

Let $\sigma$ be a 3-proper coloring of $\mathrm{Cay}(G,S)$, i.e., $\sigma\in \Sigma_G$. We will say that $\sigma$ has \emph{property $(P1)$} if for any word $w\in \Delta$, there exists a number $R_{\sigma}(w)\in \mathbf{N}$, such that every ball of radius $R_{\sigma}(w)$ in $\mathrm{Cay}(G,S)$ contains the word $w$ written along a path with the starting point marked with the color $D$ and all other points marked with the color $E$.

We say that $\sigma$ has \emph{property $(P2)$} if for every $g\in G\setminus \{e\}$, there exists a radius $N_{\sigma}(g)\in \mathbf{N}$, such that in every $N_{\sigma}(g)$-ball in $\mathrm{Cay}(G,S)$ we can find an edge $(x,y)$, such that $(x,y)\cdot g$ is also in that $N_{\sigma}(g)$-ball and the color of $(x,y)$ is different from the color of $g\cdot (x,y)$.
\end{definition}

As we will see from the next proposition, property $(P1)$ ensures that the map $\Delta \rightarrow [[G\acts M]]$ is injective, while property $(P2)$ is responsible for the freeness of the action of $G$ on $M$.

\begin{proposition}\label{prop:coloring1}
Let $G$ be a finitely generated infinite group with a finite symmetric generating set $S$. Assume that there exists a coloring $\sigma\in \Sigma_G$ with property $(P1)$. Let $M$ be a non-empty minimal closed $G$-invariant subset of the orbit closure $\overline{ G\cdot \sigma }$. Then $M$ is a Cantor space and the topological full group $[[G\acts M]]$ contains a non-abelian free group.

If we assume furthermore that $\sigma$ has property $(P2)$, then the action of $G$ on $M$ is free.
\end{proposition}

\begin{proof}
Since $M$ is a closed subspace of $\Sigma_G$, it is also compact, metrizable and totally disconnected.

Assume that there exists an isolated point $\theta\in M$. Then all points in the orbit of $\theta$ are also isolated points. Due to the minimality of $M$, the orbit of $\theta$ is dense in $M$. Since it consists of isolated points, we have $M=\overline{G\cdot \theta }=G\cdot  \theta$, so by compactness $M$ is finite.

Let us examine the property $(P1)$. The same property holds for every coloring in the orbit of $\sigma$, since $G$ acts by translations. Let $w$ be a word from $\langle A\rangle \ast \langle B \rangle \ast \langle C\rangle$. Let us take a convergent sequence $(\sigma_n)_{n\in \mathbf{N}}$ of colorings satisfying property $(P1)$, such that for all words $w$ and every $i,j\in \mathbf{N}$ we have $R_{\sigma_i}(w)=R_{\sigma_j}(w)$. In the limit coloring we will see the word $w$ written along a path with the starting point marked with $D$ and other points marked with $E$ in every ball of radius $R_{\sigma_1}(w)$, since this holds for every coloring in the sequence. Hence property $(P1)$ (with the same $R(w)$'s) holds for any coloring in the orbit closure $\overline{G\cdot \sigma}$, including the colorings in $M$.

The finiteness of the orbit $M=G\cdot \theta$ means that the coloring of some finite neighborhood of the identity element determines the whole coloring $\theta$. From property $(P1)$ it is clear that every vertex in the Cayley graph can only be the starting point of one word. So in the coloring $\theta$ we can only see finitely many words from $\langle A\rangle \ast \langle B \rangle \ast \langle C\rangle$ written along paths. This is a contradiction as $M$ cannot be finite, therefore there are no isolated points in $M$. Hence $M$ is a non-empty, compact, metrizable and totally disconnected topological space with no isolated points, so by definition it is a Cantor space.

Now consider the previously mentioned homomorphism $\Delta\to [[G\acts M]]$. Since $\Delta$ contains a non-abelian free group, it is enough to prove that this homomorphism is injective.

Let us take a word $w\in \Delta$. We would like to find an element of $M$ that is not fixed by $w$. Take an arbitrary coloring $\lambda\in M$. Due to property $(P1)$ we can see $w$ written along a marked path in the ball $B_{R_{\lambda}(w)}(e)$ in the coloring $\lambda$, say with endpoint $g\in G$. Consider the coloring $g\cdot \lambda \in M$. Here we can see $w$ written along a path ending at the origin. Since the starting point is marked with the color $D$ and the endpoint with the color $E$, this path cannot be a cycle. Therefore $w\cdot (g \cdot \lambda)\neq g \cdot \lambda$. Hence $w$ does not act as the identity on $M$, so the homomorphism $\Delta \to [[G\acts M]]$ is indeed injective.

Now assume that $\sigma$ has property $(P2)$ and assume that there exists $\lambda\in M$ and $g\in G$ such that $g\cdot \lambda=\lambda$. Let $N=N_{\sigma}(g)$ be the radius for $g$ given by property $(P2)$. Since $\lambda$ is in the closure of $G\cdot \sigma $, we can find a coloring in the orbit of $\sigma$, say $h\cdot \sigma$, such that the coloring of the $N$-ball around the identity is the same in $\lambda$ and $h\cdot \sigma$. In the coloring $h\cdot \sigma$, due to property $(P2)$, there is an edge $(x,y)$ in the $N$-ball that has a different label than $g\cdot (x,y)$, and both edges are in the $N$-ball. On this ball the coloring coincides with $\lambda$, this contradicts the assumption that $g\cdot \lambda=\lambda$. So the action of $G$ on $M$ is free.
\end{proof}

The main idea of the proof of the following proposition was communicated by Tam\'as Terpai.

\begin{proposition}\label{prop:coloring2}
Let $G$ be a finitely generated infinite group that is not virtually $\mathbf{Z}$, with symmetric generating set $S$. Then there exists a coloring $\sigma\in \Sigma_G$ satisfying the properties $(P1)$ and $(P2)$.
\end{proposition}

\begin{lemma}\label{le:coloring}
Let $G=\langle S\rangle$ be infinite but not virtually $\mathbf{Z}$. Then for any $n\in \mathbf{N}$, there exists an integer $K(n)$, such that whenever $g\in G$ and $\ell$ is a (finite or infinite) geodesic segment through $g$ in $\mathrm{Cay}(G,S)$, we can find another vertex $h$ with $\mathrm{d}(\ell,h)=n$ and $\mathrm{d}(g,h)\leq K(n)$.
\end{lemma}

\begin{proof}
We can translate the vertex $g$ to the identity element $e\in G$, so it is enough to prove the statement for $g=e$. Assume that for some $n\in \mathbf{N}$ and every $k\in \mathbf{N}$, there exists a geodesic segment $\ell_k$ through $e$ with no suitable $h$. This means that the $n$-neighborhood of $\ell_k$ covers the $k$-ball around $e$. 

By local finiteness there exists a (finite or infinite) geodesic $\ell$ through $e$, such that the $n$-neighborhood of $\ell$ covers the whole Cayley graph. 

We will estimate the growth of the Cayley graph. Let $k\in \mathbf{N}$, $k>n$, and look at $B_{k+1}(e)\setminus B_k(e)$. For $s\in \mathbf{N}$ let $\ell^{s}= \ell \cap B_{s}(e)$.
By the triangle inequality we have the following inclusions:
\[B_n(\ell^{k-n})\subseteq B_k(e) \subseteq B_{k+1}(e) \subseteq B_n(\ell^{k+n+1}).\]
Hence we have
\[B_{k+1}(e)\setminus B_k(e)\subseteq B_n(\ell^{k+n+1})\setminus B_n(\ell^{k-n}) \subseteq B_n(\ell^{k+n+1}\setminus \ell^{k-n}).\]
The set $\ell^{k+n+1}\setminus \ell^{k-n}$ has size at most $4n+2$, so the size of its $n$-neighborhood is bounded above by $(4n+2) |B_n(e)|$. This number does not depend on $k$, hence the size of $B_{k+1}(e) \setminus B_k(e)$ is bounded by a constant. Therefore the growth of the Cayley graph is at most linear, implying that the group is either finite or virtually $\mathbf{Z}$ (by \cite{Jus71}). This contradicts the assumption of the lemma, concluding the proof.
\end{proof}

\begin{proof}[Proof of Proposition \ref{prop:coloring2}]
We enumerate the words $\Delta=\{w_1,w_2,w_3,\dots \}$ and the group elements $G\setminus \{e\}=\{g_1,g_2,g_3,\dots\}$ such that for every $i\geq 1$ we have $\mathrm{length}(w_i) < \mathrm{length}(g_i)$ (we allow repetitions). We will fix ranges $\{R_i\}_{i\geq 1}$, and for every $k$, we are going to construct a coloring $\sigma_k\in \Sigma_k$ that satisfies the requirements of $(P1)$ with respect to the first $k$ words in $\Delta$, with $R_{\sigma_k}(w_i)=R_i$ for each $i\leq k$, and the requirements of $(P2)$ for the first $k$ elements of $G$, with $N_{\sigma_k}(g_i)=R_i$ for all $i\leq k$. Then we take $\sigma$ to be the limit point of a convergent subsequence of $\{\sigma_k\}_{k\geq 1}$. This $\sigma$ has property $(P1)$ and $(P2)$ with $R_{\sigma}(w_i)=R_i$ and $N_{\sigma}(g_i)=R_i$ for every $i\geq 1$.

For each $i\geq 1$, we define $R_i$ and an auxiliary range $R_i'$ the following way. Let 
$R_1'=\mathrm{length}(g_1)+2$ and
\begin{align*}
R_i' & = \max \{\mathrm{length}(g_i)+2, 2R_{i-1}'+K(2R_{i-1}'+1) \} \quad \text{for } i\geq 2,\\
R_i & = 6R_i'+K(2R_i'+1)+\mathrm{length}(g_i)+1 \quad \text{for } i\geq 1,
\end{align*}
where $K$ is the function from Lemma \ref{le:coloring}.

Now let us fix $k\in \mathbf{N}$, the construction of $\sigma_k$ is as follows. The variable $i$ will take the values $k,k-1,\dots,2,1$ successively, and for every $i$ we place copies of $w_i$, and also label additional edges that ensure property $(P2)$ for $g_i$. In each step the following three conditions will hold.
\begin{enumerate}[noitemsep, label=(\arabic*)]
\item In every $R_i$-ball we can see $w_i$ written along a path with the starting point marked with $D$, and other points marked with $E$.
\item In every $R_i$-ball there is an edge $(x,y)$ with label $D$, such that $g_i\cdot (x,y)$ has label $E$.
\item If $g$ is a vertex in a copy of the word $w_i$, then the $R_i'$-ball around $g$ does not intersect any copy of a word $w_j$ (other than the one containing $g$) for $i\leq j$ or labeled edges belonging to $g_j$ for $i < j$.
\end{enumerate}
In the $j$th step (when the value of $i$ is $k+1-j$) we take a maximal set of points $T\subset G$, such that the $2R_i'$-balls around the points of $T$
\begin{itemize}[nolistsep]
\item do not intersect each other,
\item do not intersect any previously placed copy of a word $w_j$ with $i<j$,
\item do not contain previously labeled edges belonging to $g_j$ with $i < j$.
\end{itemize}
Then for every point $t$ in $T$, we take a geodesic segment starting from $t$ and ending at $g_i\cdot t$, write $w_i$ along this segment, and mark the starting point with the color $D$, all other points with $E$. (Here we use that $\mathrm{length}(w_i) < \mathrm{length}(g_i)$.) Now for each $t\in T$, pick an edge adjacent to $t$ that is labeled by $D$, and label the $g_i$-translate of this edge by $E$, these will be the labeled edges belonging to $g_i$. This way condition (3) is immediately satisfied, since $R_i'\geq \mathrm{length}(g_i)+2$.

Take an arbitrary vertex $g\in G$, we would like to find a copy of $w_i$ and a labeled edge pair belonging to $g_i$ in the ball $B_{R_i}(g)$. If $g\in T$, then the first condition clearly holds, so assume that $g\notin T$. The reason why $g$ cannot be added to $T$ is that the $2R_i'$-ball around $g$ intersects the $2R_i'$-ball around a point of $T$, or it intersects a copy of another word $w_j$ with $i\leq j$, or a labeled edge belonging to $g_j$ with $i\leq j$.
\begin{itemize}
\item In the first case we have $p\in T$ such that $B_{2R_i'}(g)\cap B_{2R_i'}(p)\neq \varnothing$. This means that $\mathrm{d}(g,p)\leq 4R_i'$, hence the $R_i$-ball around $g$ contains the copy of $w_i$ starting from $p$, a $D$-labeled edge from $p$ and its $g_i$-translate labeled by $E$.

\item In the second case there is a point $x\in B_{2R_i'}(g)$ belonging to a copy of the word $w_j$ or an endpoint of a labeled edge belonging to $g_j$ (so $x$ is the $g_j$-translate of the starting point of a copy of $w_j$). In both cases we can consider the geodesic segment $\ell$ on which $w_j$ is written, we have $x\in \ell$. By Lemma \ref{le:coloring}, there is a point $y\in G$, such that $\mathrm{d}(y,\ell)=2R_i'+1$ and $\mathrm{d}(y,x)\leq K(2R_i'+1)$. This way the ball $B_{2R_i'}(y)$ does not intersect the geodesic $\ell$ and lies in the protective ball $B_{R_j'}(x)$ (since $i < j$), so it does not intersect any other previously placed word or labeled edge either. Hence there exists a point $q\in T\cap B_{4R_i'}(y)$, otherwise $y$ could be added to $T$ to enlarge it (possibly $y=q$). We have that 
\[\mathrm{d}(g,q)\leq \mathrm{d}(g,x)+\mathrm{d}(x,y)+\mathrm{d}(y,q)\leq \]
\[\leq 2R_i'+K(2R_i'+1)+4R_i' \leq R_i-\mathrm{length}(g_i)-1,\]
so the ball $B_{R_i}(g)$ contains a copy of $w_i$, and also an edge adjacent to $q$ labeled by $D$ and its $g_i$-translate labeled by $E$.
\end{itemize}

In both cases we can see that conditions (1) and (2) are satisfied, so we can continue the coloring. After placing all the words we label the remaining edges with the color $F$, concluding the construction of $\sigma_k$.
\end{proof}

\begin{proof}[Proof of Theorem \ref{thm:B}]
Let $S$ be a symmetric generating set for $G$. By Proposition \ref{prop:coloring2} there exists a coloring in $\Sigma_G$ with property $(P1)$ and $(P2)$. Hence, Proposition \ref{prop:coloring1} implies the statement of the theorem.
\end{proof}

\bibliographystyle{plain}
\bibliography{mybib}

\end{document}